\pdfoutput=1

\documentclass[reqno,11pt,final]{amsart}
\usepackage{glaudo_pkg_en}

\title[The optimal map in the $2$-d matching problem]
{On the optimal map in the $2$-dimensional random matching problem}
\author{L. Ambrosio}
\address{Scuola Normale Superiore, Piazza dei Cavalieri 7, 56126 Pisa, Italy}
\email{luigi.ambrosio@sns.it}
\author{F. Glaudo}
\address{ETH, Rämistrasse 101, 8092 Zürich, Switzerland}
\email{federico.glaudo@math.ethz.ch}
\author{D. Trevisan}
\address{Dipartimento di matematica, Università di Pisa, Largo Bruno Pontecorvo 56127, Pisa}
\email{dario.trevisan@unipi.it}
\begin{document}
\begin{abstract}
  We show that, on a $2$-dimensional compact manifold, the optimal transport map in the semi-discrete random 
  matching problem is well-approximated in the $L^2$-norm by identity plus the gradient of the solution 
  to the Poisson problem $-\lapl f^{n,t} = \mu^{n,t}-1$, where $\mu^{n,t}$ is an appropriate 
  regularization of the empirical measure associated to the random points.
  This shows that the ansatz of \cite{CaraccioloEtAl2014} is strong enough to capture the 
  behavior of the optimal map in addition to the value of the optimal matching cost.
  
  As part of our strategy, we prove a new stability result for the optimal transport map on a compact
  manifold.
\end{abstract}

\maketitle

\section{Introduction}\label{sec:intro}
Let $(X_1, X_2,\dots, X_n)$ be $n$ independent random points uniformly distributed on the square $\cc01^2$. 
The semi-discrete random matching problem concerns the study of the properties of the optimal 
coupling (with respect to a certain cost) of these $n$ points with the Lebesgue measure 
$\restricts{\Leb^2}{\cc01^2}$. 

More precisely, denoting $\mu^n \defeq \frac1n \sum_{i=1}^n\delta_{X_i}$ the empirical measure and 
$\m\defeq\restricts{\Leb^2}{\cc01^2}$, we want to investigate the optimal transport from 
$\m$ to $\mu^n$.

The ultimate goal is understanding both the distribution of the random variable associated to the optimal 
transport cost and the properties of the (random) optimal map. In the present paper we will show that
the optimal transport map can be well-approximated by the identity plus the gradient of the solution 
of a Poisson problem. In the large literature devoted to the matching problem,
we believe that (except for the 1-dimensional case) this is one of the few results describing
the behavior of the optimal map, and not only of the transport cost, see also 
\cite{Goldman2018} in connection with the behavior of the optimal transport map in the Lebesgue-to-Poisson problem 
on large scales. 

Before going on, let us briefly recall the definitions of optimal transport and Wasserstein distance. 
We suggest the monographs \cite{Villani08,Santambrogio15} for an introduction to the topic.
\begin{definition}[Wasserstein distance]
  Let $(X, d)$ be a compact metric space and let $\mu,\nu\in\prob(X)$ be probability measures.
  Given $p\in\co{1}{\infty}$, we define the $p$-Wasserstein distance between $\mu$ and $\nu$ as
  \begin{equation}\label{eq:defWp}
    W_p^p(\mu, \nu) \defeq \inf_{\gamma\in\Gamma(\mu, \nu)} \int_{X\times X} d^p(x, y) \de\gamma(x, y) \comma
  \end{equation}
  where $\Gamma(\mu, \nu)$ is the set of all $\gamma\in\prob(X\times X)$ such that the projections $\pi_i$, $i=1,2$, on the 
  two factors are $\mu$ and $\nu$, that is $(\pi_1)_\#\gamma = \mu$ and $(\pi_2)_\#\gamma=\nu$.
\end{definition}
\begin{remark}
  The infimum in the previous definition is always attained (\cite[Theorem 1.4]{Santambrogio15}).
  
  Moreover, if $(X, d)$ is a Riemannian manifold and $\mu\ll\m$, where $\m$ is the volume measure of 
  the manifold, the Wasserstein distance is realized by a map (\cite{McCann01}). 
  Namely, the infimum \cref{eq:defWp} is attained and the unique minimizer is induced by a Borel map $T:M\to M$, 
  so that $T_\#\mu = \nu$ and 
  \begin{equation*}
    W_p^p(\mu, \nu) = \int_M d^p(x, T(x)) \de\mu(x) \fullstop
  \end{equation*}
\end{remark}

Even though the square is a fundamental example, the random matching problem makes perfect sense even in 
more general spaces (changing the reference measure $\m$ accordingly).
Historically, in the combinatorial literature\footnote{In the combinatorial literature the problem considered
was the bipartite matching problem, in which two independent random point clouds have to be matched. The 
semi-discrete matching and the bipartite matching are tightly linked and, given that we will consider only
the former, we are going to talk about the combinatorial literature as if it were considering the semi-discrete
matching.}, the most common ambient space was $\cc01^d$ for some $d\ge 1$ and the aspect of the problem that 
attracted more attention was estimating the expected value of the $W_1$ cost. 
In the papers \cite{Ajtai1984,Talagrand1992,Dobric1995,Ledoux2017} (and possibly in other ones) 
the problem was solved in all dimensions and for all $1\le p < \infty$, obtaining the growth 
estimates\footnote{The notation $f(n)\approx g(n)$ means that there exists a positive constant $C>0$ such 
that $C^{-1}g(n) \le f(n)\le C g(n)$ for every $n$.}
\begin{equation*}
  \E{W_p^p(\m, \mu^n)} \approx 
  \begin{cases}
    n^{-\frac p2} &\text{ if $d=1$,}\\
    \left(\frac{\log(n)}n\right)^{\frac p2} &\text{ if $d=2$,}\\
    n^{-\frac pd} &\text{ if $d\ge 3$.}
  \end{cases} 
\end{equation*}
As might be clear from the presence of a logarithm, the matching problem exhibits some unexpected behavior
in dimension $2$. 

See the introductions of \cite{ambrosio-glaudo2018,Ledoux2017} or \cite[Chapter 4, 14, 15]{Talagrand14}
for a more in-depth description of the history of the problem.

Nowadays the topic is active again (\cite{holden2018,Talagrand2018,Goldman2018,Ledoux2017,ambrosio-glaudo2018,Ledoux18,Ledoux19}),
also as a consequence of \cite{Ambrosio-Stra-Trevisan2018}, in which the authors, 
following an ansatz suggested in \cite{CaraccioloEtAl2014}, manage to obtain the leading term of the 
asymptotic expansion of the expected matching cost in dimension $2$ with respect to the quadratic 
distance\footnote{The notation $f(n)\sim g(n)$ means that $\frac{f(n)}{g(n)}\to 1$ when $n\to\infty$.}:
\begin{equation}\label{eq:limit_value}
  \E{W_2^2(\m, \mu^n)} \sim \frac{\log(n)}{4\pi n} \fullstop
\end{equation}
The approach is far from being combinatorial, indeed it relies on a first-order approximation of the 
Wasserstein distance with the $H^{-1}$ negative Sobolev norm. Their proof works on any closed compact 
$2$-dimensional manifold.

Given that we will build upon it, let us give a brief sketch of the approach. 
What we are going to describe is simpler than the original approach of \cite{Ambrosio-Stra-Trevisan2018} and
can be found in full details in \cite{ambrosio-glaudo2018}. For simplicity we will assume to work on the 
square.

Let $T^n$ be the optimal map from $\m$ to $\mu^n$, whose existence is ensured by Brenier's Theorem 
(see \cite{Brenier91}). Still by Brenier's Theorem, we know that $T^n = \id + \nabla \tilde f^n$, where
$\id$ is the identity map and $\tilde f^n:\cc01^2\to M$ is a convex function. 
With high probability $\mu^n$ is well-spread on the square, thus we expect $\nabla \tilde f^n$ to be 
\emph{very small}.
We know $(T^n)_\#\m=\mu^n$ and we would like to apply the change of variable formula to deduce something 
on the Hessian of $\tilde f^n$. 
The issue is that the singularity of $\mu^n$ prevents a direct application of the change of variable 
formula. 
Anyhow, proceeding formally we obtain $\det(\id+\nabla^2 \tilde f^n)^{-1}=\mu^n$. 
Going on with the formal computation, if we consider only the first order term of 
the left hand side, the previous identity simplifies to
\begin{equation*}
  -\lapl \tilde f^n \approx \mu^n-1 \fullstop
\end{equation*}
Somewhat unexpectedly, this last equation makes perfect sense. 
Therefore we might claim that if we define $f^n:\cc01^2\to\R$ as the solution of $-\lapl f^n=\mu^n-1$
(with null Neumann boundary condition), then $T^n$ is well-approximated by $\id+\nabla f^n$ and 
furthermore the transport cost is well-approximated by $\int \abs{\nabla f^n}^2\de\m$.

This conjecture is appealing, but false, if taken literally. Indeed, it is very easy to check that the integral
$\int \abs{\nabla f^n}^2\de\m$ diverges. 

The ingredient that fixes this issue is a regularization argument. 
More precisely, let $\mu^{n,t}\defeq P_t^*\mu^n$ be the evolution at a certain small time $t>0$ of the 
empirical measure through the heat semigroup (see \cite[Chapter 6]{Chavel84}).
If we repeat the ansatz with $\mu^n$ replaced by $\mu^{n,t}$ we obtain a function $f^{n,t}:\cc01^2\to\R$ 
that solves 
\begin{equation*}
  -\lapl f^{n,t} = \mu^{n,t}-1
\end{equation*}
with null Neumann boundary conditions. Let us remark that in fact $f^{n,t}=P_t f^n$.

Once again, we can hope that $\id+\nabla f^{n,t}$ approximates very well $T^n$ and furthermore that the transport
cost from $\m$ to $\mu^n$ is well-approximated by $\int \abs{\nabla f^{n,t}}^2\de\m$.

This time the predictions are sound.
Choosing carefully the time $t=t(n)$, we can show that, with high probability, the map $\id+\nabla f^{n,t}$ is
optimal from $\m$ to $\left(\id+\nabla f^{n,t}\right)_\#\m$ and the Dirichlet energy of $f^{n,t}$ approximate
very well $W_2^2(\m, \mu^n)$.
Only one part of the conjecture is left unproven by \cite{ambrosio-glaudo2018}: is it true that 
$\id+\nabla f^{n,t}$ approximates, in some adequate sense, the optimal map $T^n$?
The goal of the present paper is to answer positively this question.

We are going to prove the following.
\begin{theorem}\label{thm:main_theorem}
  Let $(M,\metric)$ be a $2$-dimensional closed compact Riemannian manifold (or the square $\cc01^2$) whose
  volume measure $\m$ is a probability. We will denote with $d:M\times M\to\co0\infty$ the Riemannian 
  distance on $M$.
  
  Given $n\in\N$, let $X_1, X_2,\dots, X_n$ be $n$ independent random points $\m$-uniformly 
  distributed on $M$. 
  Let us denote $\mu^n\defeq \frac1n \sum_i \delta_{X_i}$ the empirical measure associated to the 
  random point cloud and let $T^n$ be the optimal transport map from $\m$ to $\mu^n$.
  
  For a fixed time $t>0$, let $\mu^{n,t}\defeq P_t^*\mu^n\in\prob(M)$ and let $f^{n,t}:M\to\R$ be the unique
  null-mean solution\footnote{If $M=\cc01^2$ we ask also that $f$ satisfies the null Neumann boundary
  conditions.} of the Poisson problem $-\lapl f^{n,t}=\mu^{n,t}-1$.
  
  If we set $t=t(n)=\frac{\log(n)^4}n$, on average $T^n$ is very close to $\exp(\nabla f^{n,t})$ in the 
  $L^2$-norm, that is
  \begin{equation}\label{eq:main-quantitative}
     \frac{\E{\int_M d^2(T^n, \exp(\nabla f^{n,t}))\de\m}}{\frac{\log(n)}{n}} \ll \sqrt{\frac{\log \left(\log (n)\right)}{ \log(n) }} \fullstop
  \end{equation}
  In particular,
  \begin{equation*}
      \lim_{n\to\infty}\frac{\E{\int_M d^2(T^n, \exp(\nabla f^{n,t}))\de\m}}
      {\E{\int_M d^2(T^n, \id)\de\m}} 
      = 0 \fullstop
  \end{equation*}
\end{theorem}
\begin{remark}
  To handle the case of the square $M=\cc01^2$ some care is required. Indeed the presence of boundary makes
  things more delicate.  This is the reason why only the square is considered in the theorem and not any
  $2$-dimensional compact manifold with boundary. 
  
  See \cite[Subsection 2.1 and Remark 3.10]{ambrosio-glaudo2018} for some further details on this matter.
\end{remark}

\begin{remark}\label{rem:distance-tangent}
By McCann’s Theorem \cite{McCann01} we can write $T^n = \exp(\nabla f^n)$, hence a natural 
question is if \cref{eq:main-quantitative} holds with $|\nabla (f^n - f^{n,t})|$ in place of 
$d(T^n, \exp(\nabla f^{n,t}))$. Using the fact that the exponential map restricted to a 
sufficiently small neighbourhood of the null vector field is a global diffeomorphism with its 
image, it would be sufficient to show that, for every 
$\varepsilon>0$, $\P{\|d(T^{n},\id)\|_\infty > \varepsilon} \ll \log(n)/n$, as $n \to \infty$. 
We will prove this estimate in \cref{prop:linf_is_small}, that provides the desired approximation at
the level of the gradients
\begin{equation}\label{eq:main_gradients}
  \lim_{n\to\infty} 
  \frac{\E{\norm{\nabla f^n-\nabla f^{n,t}}_{L^2(M)}^2}}
  {\E{\norm{\nabla f^n}_{L^2(M)}^2}} = 0\fullstop 
\end{equation}
\end{remark}

The strategy of the proof is to show that the information that we already have on $\exp(\nabla f^{n,t})$ 
(namely that it is an optimal map between $\m$ and some measure $\hat\mu^{n,t}$ that is very close to 
$\mu^{n,t}$) is enough to deduce that it must be near to the optimal map $T^n$.  

As part of the strategy of proof, we obtain, in \cref{sec:stability}, a new stability result for the optimal transport 
map on a general compact Riemannian manifold (not only of dimension $2$). 
This is the natural generalization to Riemannian manifolds of \cite{gigli2011}.
The said stability result follows rather easily from the study of the short-time behavior of the Hopf-Lax 
semigroup we perform in \cref{sec:hopflax}. 
The Hopf-Lax semigroup comes up in our investigation as, when $t=1$, it becomes the operator of
$c$-conjugation and thus produce the second Kantorovich potential once the first is known (see 
\cite[Section 1.2]{Santambrogio15} for the theory of Kantorovich potentials and $c$-conjugation).

The main theorem is established in \cref{sec:random_matching}.
\vspace{2mm}

\noindent {\textit{Acknowledgments. } F. Glaudo has received funding from the European Research Council under the 
Grant Agreement No 721675. L. Ambrosio acknowledges the support of the MIUR PRIN 2015 project.

\subsection{Notation for constants}
We will use the letters $c$ and $C$ to denote constants, whose dependencies are denoted by $c=c(A, B,\dots)$. 
The value of such constants can change from one time to the other.

Moreover we will frequently use the notation $A\lesssim B$ to hide a constant that depends only on the
ambient manifold $M$. This expression means that there exists a constant $C=C(M)$ such that $A\le C\cdot B$.

\section{Short-time behavior of the Hopf-Lax semigroup with datum in \texorpdfstring{$C^{1,1}$}{C(1,1)}}
\label{sec:hopflax}
Let us begin recalling the definition of the Hopf-Lax semigroup (also called Hamilton-Jacobi semigroup).
\begin{definition}[Hopf-Lax semigroup]
  Let $(X, d)$ be a compact length space\footnote{A metric space is a length space if the distance between
  any two points is the infimum of the length of the curves between the two points. Let us remark that for
  the definition we need neither the compactness nor the length property of $X$, but without these 
  assumptions many of the properties of the Hopf-Lax semigroup fail (first of all the fact that it is 
  a semigroup).}.
  For any function $f\in C(X)$ and any $t\geq 0$, let $Q_t f:X\to\R$ be defined by
  \begin{equation*}
    Q_t f(y) = \min_{x\in X} \frac1{2t} d^2(x,y) + f(x) \quad (t>0),\qquad Q_0f=f\fullstop
  \end{equation*}
\end{definition}
Without additional assumptions on $X$ or $f$ it is already possible to deduce many properties of the Hopf-Lax
semigroup. Let us give a very short summary of the most important ones.
\begin{itemize}
  \item When $t\to 0$ the functions $Q_t f$ converge uniformly to $f$.
  \item The Hopf-Lax semigroup is indeed a semigroup, that is $Q_{s+t}f = Q_sQ_t f$ for any $s,\, t \geq 0$.
  \item In a \emph{suitable weak sense},the Hamilton-Jacobi equation
  \begin{equation*}
    \frac{\de}{\de t} Q_t f + \frac12\abs{\nabla Q_t f}^2 = 0 
     \end{equation*}
      holds.
  Let us emphasize that the mentioned equation does not make sense if we don't give an appropriate definition
  of norm of the gradient as we are working in a metric setting.
\end{itemize}
See \cite{Lott-Villani07}, in particular Theorem~2.5, for a detailed proof of the mentioned
properties. 

There is a vast literature investigating the regularity properties of the Hopf-Lax semigroup and its 
connection with the Hamilton-Jacobi equation, in particular that it is the unique solution in the viscosity
sense (see for instance \cite{lions1982,benton1977,bardi2008}). 
Nonetheless we could not find a complete reference for the short-time behavior of the Hopf-Lax semigroup on 
a Riemannian manifold (as the majority of the results are stated on the Euclidean space) with a relatively 
regular initial datum (namely $C^{1,1}$). This is exactly the topic of this section. 

What we are going to show, apart from \cref{it:hopflax_convexity}, is not new. For instance,
in \cite[Section 5]{fathi2003}, the author proves the validity of the method of characteristics in 
a way very similar to ours. In that paper more general Lagrangians are considered and as a consequence 
the proofs are more involved and require much more geometric tools and notation.

For us, the ambient space is a compact Riemannian manifold $(M, \metric)$ and the function 
$f\in C^{1,1}(M)$ is differentiable with Lipschitz continuous gradient. Moreover, either $M$ is closed
or it is the square $\cc01^2$. For a general
manifold with boundary the results are false, the square is special because its boundary is piecewise geodesic.
Handling all manifolds with totally geodesic boundary would be possible, but would require some additional 
care. In order to simplify the exposition we decided to state the results only for the square.
Throughout this section we will often use implicitly that a Lipschitz continuous function is differentiable
almost everywhere (see \cite[Theorem 3.2]{Evans-Gariepy}).

We will show that, up to a small time that depends on the $C^{1,1}$-norm of $f$, the Hopf-Lax semigroup 
is \emph{as good as one might hope}.
We will describe explicitly the minimizer $x=x_t(y)$ of the variational problem that 
defines $Q_t f(y)$ deducing some \emph{explicit} formulas for $Q_tf$ and its gradient and we will
show that $Q_t f$ solves the Hamilton-Jacobi equation in the classical sense.
Finally we will be able to control the $C^{1,1}$-norm of $Q_t f$ and the $C^{0,1}$-norm of $Q_t f-f$. 

How can we achieve these results for short times when $f\in C^{1,1}$? 
The main ingredient is the possibility to identify the minimizer $x=x_t(y)$ in the definition of $Q_t f(y)$.
Given $x\in M$, let $\gamma:\co{0}{\infty}\to M$ be the unique geodesic with $\gamma(0) = 0$ and 
$\gamma'(0)=\nabla f(x)$. If $y=\gamma(t)$, then the minimizer in the definition of $Q_t f(y)$ is exactly $x$.
This approach is exactly the method of characteristics when applied on a Riemannian manifold (\emph{straight 
lines on a manifold are geodesics}).

Let us begin with a technical lemma. 
\begin{lemma}\label{lem:exp_is_diffeo}
  Let $(M, \metric)$ be a closed compact Riemannian manifold (or the square $\cc01^2$).

  There exists a constant $c=c(M)$ such that the following statement holds.
  Let $X\in\chi(M)$ be a Lipschitz continuous vector 
  field\footnote{If $M=\cc01^2$ we ask also that $X$ is tangent to the boundary.}
  with $\norm{X}_\infty\le c$ and $\norm{\nabla X}_\infty\le c$ and, for any $0\le t\le 1$, let 
  $\varphi_t:M\to M$ be the map defined as $\varphi_t(x) \defeq \exp(tX(x))$, where $\exp:TM\to M$ denotes the 
  exponential map.
  For any $0\le t\le 1$, the map $\varphi_t$ is a homeomorphism 
  such that $\Lip(\varphi_t)$, $\Lip(\varphi_t^{-1}) \le 2$ and the vector field $X_t\in\chi(M)$ defined
  as
  \begin{equation*}
    X_t \defeq \frac{\partial \varphi_s}{\partial s}\Big|_{s=t}
  \end{equation*}
  is Lipschitz continuous with $\norm{\nabla X_t}_{\infty}\lesssim\norm{\nabla X}_{\infty}$.
\end{lemma}
\begin{proof}
  We will give only a sketch of the proof of the first part of the statement as the argument is well-known.

  Let us begin by proving the result when $M$ is closed (in particular we exclude only $M=\cc01^2$).
  
  We can deduce the first part of the statement from the fact that $\varphi=\varphi_1$ is injective and 
  locally (i.e. on sufficiently small balls) it is a bi-Lipschitz transformation with its image.
  
  Working in a suitably chosen finite atlas (whose existence follows from the compactness of $M$), the fact 
  that $\varphi$ is a bi-Lipschitz diffeomorphism is a consequence of the following very 
  well-known lemma about perturbations of the identity (see \cite[Theorem 9.24]{rudin1976} or 
  \cite[Theorem 5.3]{fathi2003}).
  If $T:\Omega\subseteq \R^d\to\R^d$ is such that $T-\id$ is $L$-Lipschitz with $L<1$, then $T$ is locally
  invertible and $\Lip(T) \le 1+L,\ \Lip(T^{-1}) \le (1-L)^{-1}$.
  
  The global injectivity follows directly from the fact that it is locally bi-Lipschitz. Indeed if 
  $\varphi(x_1)=\varphi(x_2)$ then $d(x_1, x_2) \le 2\norm{X}_\infty$ and therefore we can exploit the local
  injectivity of $\varphi$.
  
  When $M=\cc01^2$ we need only a simple additional remark. Given that $X$ is tangent to the boundary, the map 
  $\varphi$ is a homeomorphism of the boundary. As a consequence of this fact, it is not
  difficult to prove (by injectivity) that the image of the interior of the square is mapped 
  by $\varphi$ in itself. 
  From here on we can simply mimic the proof described above for closed manifolds and achieve the
  result also for the case of the square.
  
  We move our attention to the second part of the statement.
  By a simple homogeneity argument, it is sufficient to prove that
  $\norm{\nabla X_t}_{\infty}\lesssim 1$.
  
  Once again we work in chart. Let $\Omega\subseteq\R^d$ be the domain of the chart. 
  As usual, $X_t$ can be understood as a vector field on $\Omega$ and $\varphi_t$ as a map from 
  $\Omega'\Subset\Omega$ into $\Omega$.
  Choosing the chart appropriately, we can assume that the Euclidean distance is bi-Lipschitz equivalent
  to the distance induced by the metric $\metric$.
  
  The Lipschitz continuity of $X_t$ with respect to the metric $\metric$ is equivalent to proving that, for 
  any $x,y\in\Omega$, it holds
  \begin{equation*}
    \abs{X_t(x)-X_t(y)} \lesssim \abs{x-y} \comma
  \end{equation*}
  where all the absolute values are with respect to the standard Euclidean norm. 
  Since $\varphi_t$ is surjective, it is sufficient to prove that, for any $x,y\in\Omega'$, it holds
  \begin{equation}\label{eq:exp_diffeotmp1}
    \abs{X_t(\varphi_t(x))-X_t(\varphi_t(y))} \lesssim \abs{\varphi_t(x)-\varphi_t(y)} \fullstop
  \end{equation}
  
  Given that $\varphi_t^{-1}$ is Lipschitz, we already know
  \begin{equation}\label{eq:exp_diffeotmp2}
    \abs{x-y} \lesssim \abs{\varphi_t(x)-\varphi_t(y)} 
    \quad\text{and}\quad
    \abs{X(x)-X(y)} \lesssim \abs{\varphi_t(x)-\varphi_t(y)} \fullstop
  \end{equation}
  Let $\gamma_x:\cc01\to\Omega$ be the unique geodesic, with respect to $\metric$, such that $\gamma_x(0)=x$
  and $\gamma_x'(0)=X(x)$. Let $\gamma_y:\cc01\to\Omega$ be defined analogously. By definition, it holds
  \begin{equation}\label{eq:exp_diffeotmp3}
    X_t(\varphi_t(x)) = \gamma_x'(t)
    \quad\text{and}\quad
    X_t(\varphi_t(y)) = \gamma_y'(t) \fullstop
  \end{equation}
  Taking into account \cref{eq:exp_diffeotmp1}, \cref{eq:exp_diffeotmp2} and \cref{eq:exp_diffeotmp3}, the 
  Lipschitz continuity of $X_t$ would follow from the inequality
  \begin{equation}\label{eq:exp_diffeotmp4}
    \abs{\gamma_x'(t)-\gamma_y'(t)}
    \lesssim \abs{\gamma_x(0)-\gamma_y(0)} + \abs{\gamma_x'(0)-\gamma_y'(0)}
    \fullstop
  \end{equation}
  The curves $\gamma_x,\gamma_y$ are geodesics, hence the vectors $(\gamma_x, \gamma_x')$ and 
  $(\gamma_y,\gamma_y')$ solve the same autonomous ordinary differential equation with different initial data.
  Hence \cref{eq:exp_diffeotmp4} follows from the well-known Lipschitz dependence of the solution 
  from the initial data (see \cite[Theorem 2.6]{teschl2012}) and therefore the proof is concluded.
\end{proof}

We can now state and prove the main theorem of this section. 
The technically demanding part of these notes is entirely enclosed in the following theorem.
\begin{theorem}\label{thm:hopflax_properties}
  Let $(M, \metric)$ be a closed compact Riemannian manifold (or the square $\cc01^2$).
  
  Let $f\in C^{1,1}(M)$ be a scalar function\footnote{If $M=\cc01^2$ we ask also that $f$ satisfies the 
  null Neumann boundary conditions.} and, for any positive time $t>0$, let us define the map 
  $\varphi_t:M\to M$ as $\varphi_t(x) \defeq \exp(t\nabla f(x))$.
  
  There exists a constant $c=c(M)$ such that the following properties hold for any time 
  $0\le t\le c \left(\norm{\nabla f}_\infty + \norm{\nabla^2 f}_\infty\right)^{-1}$:
  \begin{enumerate}[ref={(\arabic*)}]
   \item \label{it:varphi_t_diffeo}
   The map $\varphi_t$ is a bi-Lipschitz homeomorphism such that 
   $\Lip(\varphi_t),\, \Lip(\varphi_t^{-1}) \leq 2$.
   \item \label{it:hopflax_explicit} For any $y\in M$, it holds
   \begin{equation*}
      Q_t f(y) = \frac1{2t} d^2(\varphi_t^{-1}(y), y) + f(\varphi_t^{-1}(y)) \fullstop
   \end{equation*}
   \item \label{it:hopflax_convexity}
   For any $y,\,y'\in M$, one has the (strict-convexity-like) estimate
   \begin{equation*}
      \frac{d^2(y, y')}t \lesssim Q_tf(y)-Q_tf(y')
      +\frac1{2t}\left[d^2(\varphi_t^{-1}(y), y') - d^2(\varphi_t^{-1}(y), y)\right]\fullstop
   \end{equation*}
   \item \label{it:hopflax_regularity}
   The function $Q_tf$ is Lipschitz continuous in time and $C^{1,1}(M)$ in space. In particular we have
   $\norm{\partial_t Q_t f}_{\infty}\le \norm{\nabla f}_{\infty}$ and
   $\norm{\nabla^2Q_tf}_{\infty} \lesssim \norm{\nabla^2 f}_{\infty}$.
   \item \label{it:hamilton_jacobi} 
   The function $Q_tf$ is a classical solution of the Hamilton-Jacobi equation
   \begin{equation*}
      \frac{\de}{\de t} Q_t f + \frac12 \abs{\nabla Q_t f}^2 = 0 \fullstop
   \end{equation*}
   \item \label{it:gradient_conservation} 
   For any $x\in M$, if $\gamma:\cc01\to M$ is the geodesic such that $\gamma(0)=x$ and 
   $\gamma'(0) = \nabla f(x)$, then it holds
   \begin{equation*}
      Q_tf(\gamma(t)) = f(x) + \frac t2 \abs{\nabla f}^2(x) \quad\text{ and }\quad
      \nabla Q_t f(\gamma(t)) = \gamma'(t) \fullstop
   \end{equation*}
   \item \label{it:hopflax_lip}
   One has
   \begin{equation*}
    \Lip(Q_tf-f) \le t\norm{\nabla f}_\infty\cdot\norm{\nabla^2 f}_\infty \fullstop
  \end{equation*}
  \end{enumerate}
\end{theorem}
\begin{proof}
  Thanks to the following homogeneity, for any $t>0$ and $\lambda>0$, of the Hopf-Lax semigroup
  \begin{equation*}
    Q_t(\lambda f)(y) = \lambda Q_{\lambda t} f(y) \comma
  \end{equation*}
  we can assume without loss of generality that $\norm{\nabla f}_\infty + \norm{\nabla^2 f}_\infty \le c$ 
  and prove that the statements hold up to time $1$.
  Thus, we will implicitly assume that the time variable satisfies $0\le t\le 1$.
  We will choose the value of the constant $c$ during the proof, it should be clear that all constraints 
  we impose depend only on the manifold $M$ and not on the function $f$.

  The statement of \cref{it:varphi_t_diffeo} follows from \cref{lem:exp_is_diffeo}.
  
  To prove \cref{it:hopflax_explicit} we need some preliminary observations. 
  If $c=c(M)$ is sufficiently small (so that the constraint on $f$ is sufficiently strong), thanks to 
  the compactness of $M$ we can find a radius $r=r(M)>0$ such that:
  \begin{enumerate}[label=(\alph*)]
   \item If $p,\,q\in M$ satisfy $d(p,q)\le r$ then
   \begin{equation*}
      \nabla^2 d^2(\emptyparam,p)(q) \ge \frac12 \metric \fullstop
   \end{equation*}
   \item For any $y\in M$, to compute $Q_tf(y)$ it is sufficient to minimize on $B(y,r)$:
   \begin{equation*}
      Q_tf(y) = \inf_{x\in B(y, r)} \frac1{2t}d^2(x, y)+f(x) \fullstop
   \end{equation*}
   \item For any $y\in M$ it holds the inequality $d(y, \varphi_t^{-1}(y)) \le r$. 
   In particular we can assume that $\varphi_t^{-1}(y)$ is not in the cut-locus of $y$.
   \item For any $y\in M$ it holds the identity
   \begin{equation*}
      \nabla \left(\frac1{2t}d^2(\emptyparam, y) + f(\emptyparam)\right)(\varphi_t^{-1}(y)) = 0 \fullstop
   \end{equation*}
   This identity can be shown computing the gradient of the distance from $y$ squared, since we know
   that $y=\exp(t\nabla f(x))$ where $x=\varphi_t^{-1}(y)$. Indeed, given that $x$ does not belong to the 
   cut-locus of $y$, we know
   \begin{equation*}
      \nabla \left(\frac12 d^2(\emptyparam, y)\right)(x) = -t\nabla f(x)
   \end{equation*}
   and the desired identity follows.
  \end{enumerate}
  With these observations at our disposal, the proof of \cref{it:hopflax_explicit} is straight-forward.
  Given a time $0\le t\le 1$ and a point $y\in M$, let us consider the function $w_{t,y}:M\to\R$ defined
  as
  \begin{equation*}
    w_{t,y}(x) = \frac1{2t}d^2(x,y) + f(x) \fullstop
  \end{equation*}
  We know that $Q_tf(y) = \min_{x\in B(y, r)} w_{t,y}(x)$. Moreover $\nabla w_{t,y}(\varphi_t^{-1}(y))=0$ and,
  if the constraint on $\norm{\nabla^2 f}_\infty$ is sufficiently small, we also know 
  $\nabla^2 w_{t,y}\ge \frac1{3t}\metric$ in $B(y, r)$. 
  Hence, by convexity, we deduce that $\varphi_t^{-1}(y)$ is the global minimum point of 
  $w_{t,y}$ and \cref{it:hopflax_explicit} follows.
  
  Let us now move to the proof of \cref{it:hopflax_convexity}. 
  Let $x,\, x'\in M$ be such that $\varphi_t(x) = y$ and $\varphi_t(x')=y'$. 
  Applying \cref{it:hopflax_explicit} and recalling that $\varphi_t$ is a bi-Lipschitz 
  diffeomorphism, we can see that the inequality we want to prove is equivalent to
  \begin{equation*}
    \frac1t d^2(x, x') \lesssim f(x)-f(x') + \frac 1{2t}\left(d^2(x, y')-d^2(x', y') \right)
  \end{equation*}
  and, using the same notation as above, this becomes
  \begin{equation*}
    \frac1t d^2(x, x') \lesssim w_{t, y'}(x) - w_{t, y'}(x') \fullstop
  \end{equation*}
  The latter inequality follows from the strict convexity of $w_{t,y'}$ that we have already shown 
  while proving \cref{it:hopflax_explicit}.

  Showing from scratch that $Q_tf$ solves the Hamilton-Jacobi equation would not be hard, but for this 
  we refer to \cite[Theorem 2.5, viii]{Lott-Villani07}, where the authors show that $Q_tf$ is 
  a \emph{suitably weak} solution of the Hamilton-Jacobi equation. 
  From their statement, we can deduce that if $Q_tf$ is differentiable at $x\in M$, then
  \begin{equation}\label{eq:hamilton_jacobi_ae}
    \frac{\de}{\de t}Q_tf(x) + \abs{\nabla Q_tf(x)}^2 = 0 \fullstop
  \end{equation}
  Since we will show that $Q_tf$ is $C^{1,1}(M)$, the validity of 
  \cref{it:hopflax_regularity,it:hamilton_jacobi} is a consequence of \cref{eq:hamilton_jacobi_ae}.
  
  The first part of \cref{it:gradient_conservation}, namely 
  $Q_tf(\gamma(t)) = f(x) + \frac t2\abs{\nabla f}^2(x)$, is implied by \cref{it:hopflax_explicit}. 
  To obtain the identity involving the gradient, let us differentiate the previous equality with respect to the
  time variable. If $Q_t f$ is differentiable at $\gamma(t)$, it holds
  \begin{equation}\label{eq:almost_gradient_conservation}
    \frac{\de}{\de t} (Q_t f)(\gamma(t)) + \scalprod{\nabla Q_t f(\gamma(t))}{\gamma'(t)} = 
    \frac{\de}{\de t} (Q_tf(\gamma(t))) = \frac12\abs{\nabla f}^2(x)\fullstop
  \end{equation}
  Applying \cref{it:hamilton_jacobi} and the fact that $\abs{\gamma'(t)} = \abs{\nabla f}(x)$, from
  \cref{eq:almost_gradient_conservation} we can deduce
  \begin{equation}\label{eq:hopflax_tmp}
    -\frac12 \abs{\nabla Q_tf}^2(\gamma(t)) + \scalprod{\nabla Q_t f(\gamma(t))}{\gamma'(t)} 
    = \frac12\abs{\gamma'}^2(x) \iff \abs{\nabla Q_tf(\gamma(t)) - \gamma'(t)}^2 = 0 \fullstop
  \end{equation}
  This does not imply directly \cref{it:gradient_conservation} since we have shown the identity only if
  $Q_t f$ is differentiable at $\gamma(t)$.
  As a byproduct of \cref{it:hopflax_explicit}, we know that $Q_t f$ is Lipschitz continuous and 
  therefore, from \cref{eq:hopflax_tmp}, we can deduce that, fixed $t$, for almost every $x\in M$ it holds
  \begin{equation*}
    \nabla Q_t f(\varphi_t(x)) = \frac{\partial \varphi_s(x)}{\partial s}\Big|_{s=t} \fullstop
  \end{equation*}
  Since the right-hand side is Lipschitz continuous (see \cref{lem:exp_is_diffeo}) it follows that  
  $Q_tf\in C^{1,1}(M)$ and, as anticipated, this concludes the proofs of 
  \cref{it:hopflax_regularity},\cref{it:hamilton_jacobi} and \cref{it:gradient_conservation}.
  
  Finally let us tackle \cref{it:hopflax_lip}.
  Given $y\in M$, let $x=\varphi_t^{-1}(y)$. Thanks to \cref{it:gradient_conservation}, if we consider 
  the geodesic $\gamma:\cc01\to M$ such that $\gamma(0)=x$ and $\gamma'(0)=\nabla f(x)$, we know that 
  $\gamma(t)=y$ and $\gamma'(t)=\nabla Q_tf(y)$.
  
  Thus we have
  \begin{equation*}
    \abs{\nabla f(y) -\nabla Q_tf(y)} 
    \le \int_0^t \abs{\nabla_{\gamma'}\left(\nabla f(\gamma)-\gamma'\right)}\de s
    \le t\abs{\nabla f(x)}\cdot\norm{\nabla^2 f}_\infty
  \end{equation*}
  and this is the desired statement.
\end{proof}

\begin{remark}
  Let us emphasize that the only statement contained in \cref{thm:hopflax_properties} that we are going to
  use is \cref{it:hopflax_convexity}. Indeed it will be crucial when studying the stability of optimal
  maps. Furthermore, such a statement should be seen more like as a property of the $c$-conjugate 
  (see \cite[Section 1.2]{Santambrogio15}) than as a property of the Hopf-Lax semigroup.
  
  We have proven all other statements in order to give a complete reference on the short-time behavior of
  the Hopf-Lax semigroup when the initial datum is in $C^{1,1}(M)$.
\end{remark}

\section{Quantitative Stability of the Optimal Map}\label{sec:stability}
In this section we will always refer to the optimal transport with respect to the quadratic
cost between two probability measures in $\prob(M)$ that are absolutely continuous with respect to the volume
measure $\m$ of a compact Riemannian manifold $(M, \metric)$.

The duality theory of optimal transport can be seen as a tool to bound from above and from below 
the optimal transport cost. Indeed, simply producing a transport map we can bound the cost from above, whereas
with a pair of potentials we can bound it from below. Estimating the optimal cost is the best one can 
desire for a generic convex problem, but for the optimal transport problem we know that the optimal 
map is unique (see \cite{McCann01}) and thence we would like to be able to approximate it.

In details, we want to investigate the following problem.
\begin{problem}
  Let $\nu, \,\mu_1,\, \mu_2\in \prob(M)$ be probability measures with $\nu\ll \m$. Let $S,\, T$ be 
  the optimal transport maps from $\nu$ to $\mu_1$ and $\mu_2$ respectively. 
  Estimate the $L^2(\nu)$-distance $\norm{d(S, T)}^2_{L^2(\nu)}$ between the two maps.
\end{problem}
The approach we are going to adopt builds upon the method, suggested to N.Gigli by the first author, who
used it in \cite[Proposition 3.3 and Corollary 3.4]{gigli2011}. 
In the proof of the mentioned results, the author obtains (even if not stated in this way) exactly
the same inequality we are going to obtain. The substantial difference is that those results (and their proofs)
work only when the ambient is the Euclidean space. 

Transporting the proofs from the flat to the curved setting is not straight-forward. The proof of 
Proposition~3.3 of the mentioned paper does not work on a Riemannian manifold, 
because curvature comes into play when comparing tangent vectors at different points.
To overcome this difficulty we have
come up with \cref{it:hopflax_convexity} of \cref{thm:hopflax_properties}. On the contrary, the proof of 
Corollary 3.4 is easily adapted on a compact Riemannian manifold.

Let us also mention the recent result \cite[Theorem 4.1]{berman2018}. In the said theorem the author
obtain a quantitative stability of the optimal map when, instead of changing the target measure as we are
doing, the source measure is changed. The proof is totally different from ours and is mainly based on
complex analytic tools. Also in that paper only the Euclidean setting (and the flat torus) is considered.

We will attack the stability problem only in the \emph{perturbative setting}, namely when the optimal
map from $\nu$ to $\mu_1$ is the identity up to the first order.
Working only in the perturbative setting might look like an extremely strong assumption that would yield
no applications at all. This is not the case, indeed what we call \emph{perturbative setting} is 
more or less equivalent to requiring only that the optimal transport map $T$ is local (meaning that 
$T-\id$ is uniformly small) and well-behaved.
For example, and this is the whole point of \cite{ambrosio-glaudo2018}, the optimal map from the reference 
measure to a random point cloud is (with high probability) a perturbation of the identity.

We don't need any hypothesis on the optimal map between $\nu$ and $\mu_2$.

\begin{theorem}\label{thm:optimal_map_stability}
  Let $(M,\metric)$ be a closed compact Riemannian manifold (or the square $\cc01^2$) and let us 
  denote by $\m$ its volume measure.

  Let $\nu, \mu_1, \mu_2\in\prob(M)$ be three probability measures with $\nu\ll\m$ and let $S, T:M\to M$ 
  be the optimal transport maps respectively for the pairs of measures $(\nu, \mu_1)$ and $(\nu, \mu_2)$. 
  We assume that $S=\exp(\nabla f)$ where $f:M\to\R$ is a 
  $C^{1,1}$-function\footnote{If $M=\cc01^2$ we ask also that $f$ satisfies the null Neumann boundary
  conditions.} such that $\norm{\nabla f}_\infty + \norm{\nabla^2 f}_\infty \le c$ where $c=c(M)$ is 
  the constant considered in the statement of \cref{thm:hopflax_properties}.
  
  Then it holds
  \begin{equation*}
    \int_M d^2(S, T)\de\nu \lesssim W_2^2(\mu_1, \mu_2) + W_2(\mu_1, \mu_2)W_2(\nu, \mu_1) \fullstop
  \end{equation*}
\end{theorem}
\begin{proof}
  Let us consider a generic transport map $S':M\to M$ from $\nu$ to $\mu_1$ and recall that,
  according to \cite{Glaudo19}, if $c(M)$ is small enough, then the map $S$ is optimal.

  Given $x\in M$, let us apply \cref{it:hopflax_convexity} of \cref{thm:hopflax_properties} with 
  $y=S(x)$ and $y'=S'(x)$ and $t=1$
  \begin{equation*}
     d^2(S(x),S'(x)) 
      \lesssim Q_1f(S(x))-Q_1f(S'(x)) + \frac12\left(d^2(x, S'(x))-d^2(x, S(x))\right) \fullstop
  \end{equation*}
  Integrating this inequality with respect to $\nu$ we obtain
  \begin{equation}\label{eq:stability_same_measures}
    \int_M d^2(S, S')\de\nu \lesssim \norm{d(S', \id)}_{L^2(\nu)}^2 - W_2^2(\nu, \mu_1)
  \end{equation}
  as the first two terms cancel thanks to the fact that both $S$ and $S'$ sends $\nu$ into $\mu_1$.
  
  We can now prove the main statement under the additional assumption that there exists an optimal map 
  $R:M\to M$ from $\mu_2$ to $\mu_1$. Applying \cref{eq:stability_same_measures} with $S' = R\circ T$ we get
  \begin{equation}\label{eq:temporary_ineq}
    \int_M d^2(S, R\circ T)\de\nu \lesssim \norm{d(R\circ T, \id)}_{L^2(\nu)}^2 - W_2^2(\nu, \mu_1) \fullstop
  \end{equation}
  Thanks to the triangle inequality, it holds
  \begin{align*}
    \norm{d(R\circ T, \id)}_{L^2(\nu)} 
    &\le \norm{d(R\circ T, T)}_{L^2(\nu)} + \norm{d(T, \id)}_{L^2(\nu)}
    = \norm{d(R, \id)}_{L^2(\mu_2)} + W_2(\nu, \mu_2) \\
    &\le 2 W_2(\mu_1, \mu_2) + W_2(\nu, \mu_1) \fullstop
  \end{align*}
  Applying this last inequality into \cref{eq:temporary_ineq} yields
  \begin{align*}
    \int_M d^2(S, R\circ T)\de\nu 
    &\lesssim \left[2 W_2(\mu_1, \mu_2) + W_2(\nu, \mu_1)\right]^2 - W_2^2(\nu, \mu_1) \\
    &\lesssim W_2^2(\mu_1, \mu_2) + W_2(\mu_1, \mu_2)W_2(\nu, \mu_1)
  \end{align*}
  and the desired statement follows from the triangle inequality
  \begin{align*}
    \int_M d^2(S, T) &\lesssim \int_M d^2(S, R\circ T)\de\nu + \int_M d^2(R\circ T, T)\de\nu \\
      &\lesssim W_2^2(\mu_1, \mu_2) + W_2(\mu_1, \mu_2)W_2(\nu, \mu_1) + \int_M d^2(R, \id)\de\mu_2 \\
      &= 2 W_2^2(\mu_1, \mu_2) + W_2(\mu_1, \mu_2)W_2(\nu, \mu_1) \fullstop
  \end{align*}

  It remains to drop the assumption on the existence of the optimal map $R$. Given that our ambient manifold 
  is compact, we can apply the nonquantitative strong stability (see \cite[Corollary 5.23]{Villani08}). 
  Let us take a sequence of absolutely continuous probability measures $\mu_2^n$ that weakly converges 
  to $\mu_2$. 
  Thanks to McCann's Theorem (see \cite{McCann01}) the optimal map $R^n$ from $\mu_2^n$ to $\mu_1$ exists 
  and thanks to the strong stability we know that the optimal maps $T^n$ from $\nu$ to $\mu_1^n$ converge 
  strongly in $L^2(\nu)$ to $T$. 
  Hence it is readily seen that the result for $\mu_2$ can be obtained by passing to the limit the 
  result for $\mu_2^n$. 
\end{proof}

\begin{remark}
  The first part of the proof of \cref{thm:optimal_map_stability} might seem a bit magical. Let us describe
  what is happening under the hood. 
  
  The function $f$ is the Kantorovich potential of the couple $(\nu, \mu_1)$ and
  hence, by standard theory in optimal transport, it must be $c$-concave. 
  
  Our hypotheses ensure us that it is not only $c$-concave, but even \emph{strictly} $c$-concave. 
  Furthermore, the theory we have developed on the Hopf-Lax semigroup tells us that even the other 
  potential $f^c=Q_1 f$ is strictly $c$-concave (this is exactly \cref{it:hopflax_convexity}).
  
  The result follows integrating the strict $c$-concavity inequality with respect to the measure
  $\nu$.
\end{remark}

\begin{remark}
  The main use of \cref{thm:optimal_map_stability} is the following one. 
  Assume that the optimal map from $\nu$ to $\mu_1$ is local and well-behaved (this ensures the validity 
  of the hypotheses of the theorem) and furthermore that $\mu_2$ is much closer to $\mu_1$ than to $\nu$. 
  In this situation, the theorem tells us
  \begin{equation*}
    \int_M d^2(S, T)\de\nu \ll \int_M d^2(S, \id)\de\nu \comma
  \end{equation*}
  and this conveys exactly the information that $S$ approximates very well $T$. Notice also
  that the improvement from $C^{0,1/2}$ dependence of \cite{gigli2011} to the kind of Lipschitz dependence 
  is due to the fact that we are working in a perturbative regime, close to the reference measure.
\end{remark}

\section{Optimal map in the random matching problem}\label{sec:random_matching}
We want to apply our result on the stability of the optimal map in the perturbative setting to the 
semi-discrete random matching problem. In this section we will work on a compact closed Riemannian manifold 
$(M, \metric)$ of dimension $2$ (or the square $\cc01^2$). 
We will denote with $\m$ the volume measure, with the implicit assumption that it is a probability.

In this setting, the semi-discrete random matching problem can be formulated as follows.
For a fixed $n\in\N$, consider $n$ independent random points $X_1, X_2, \dots, X_n$ $\m$-uniformly distributed
on $M$. Study the optimal transport map $T^n$ (with respect to the quadratic cost) from $\m$ to the empirical 
measure $\mu^n = \frac1n\sum_{i} \delta_{X_i}$.

Since we want to attack the problem applying \cref{thm:optimal_map_stability}, first of all we have to choose 
$\nu,\, \mu_1$ and $\mu_2$.
The choices of $\nu$ and $\mu_2$ are very natural, indeed we set $\nu=\m$ and $\mu_2=\mu^n$. This way the map
$T$ is $T^n$ .

Far less obvious is the choice of $\mu_1$, $S$ and $f$. As one might expect from the statement 
of \cref{thm:main_theorem} and from the ansatz described in the introduction, our choice is $f=f^{n,t}$.
Thus $S=\exp(\nabla f^{n,t})$ (for some appropriate $t=t(n)$). 
Furthermore, keeping the same notation of \cite{ambrosio-glaudo2018}, the measure $\mu_1=S_\#\m$ will 
be denoted by $\hat\mu^{n,t}$.

First of all it is crucial to understand whether we are in position to apply \cref{thm:optimal_map_stability}.
Indeed we need to check if $\nabla^2 f^{n,t}$ and $\nabla f^{n,t}$ are sufficiently small. Moreover we have
to obtain a strong estimate on $W_2^2(\mu_1, \mu_2)$. 
Both this facts are among the main results obtained in \cite{ambrosio-glaudo2018}. 
Hence let us state them in the following proposition.

\begin{proposition}[Summary of results from \texorpdfstring{\cite{ambrosio-glaudo2018}}{[AG18]}]\label{prop:ag18}
  Let $(M, \metric)$ be a closed compact $2$-dimensional Riemannian manifold (or the square $\cc01^2$) whose
  volume measure $\m$ is a probability.
  Given $n\in\N$, let $X_1,\dots, X_n$ be $n$ independent random points $\m$-uniformly distributed on $M$
  and denote $\mu^n=\frac1n\sum_i \delta_{X_i}$ the associated empirical measure.
  
  For a choice of the time $t>0$, let $\mu^{n,t}=P^*_t(\mu^n)$ be the evolution through the heat flow of
  the empirical measure and let $f^{n,t}:M\to\R$ be the unique null-mean solution\footnote{If $M=\cc01^2$
  we ask also that $f$ satisfies the null Neumann boundary conditions.} to the Poisson equation
  $-\lapl f^{n,t} = \mu^{n,t}-1$.
  Finally, let us define the probability measure $\hat\mu^{n,t}$ as the push-forward of $\m$ through the map
  $\exp(\nabla f^{n,t})$.
  
  For any $\xi>0$, let $A^{n,t}_\xi$ be the probabilistic event $\{\norm{\nabla^2 f^{n,t}}_\infty < \xi\}$.
  
  If $t=t(n)=\frac{\log^4(n)}{n}$ and $\xi=\xi(n) = \frac1{\log(n)}$, the following statements\footnote{In 
  \cite{ambrosio-glaudo2018} the time $t(n)$ is chosen as $t(n)=\gamma\frac{\log^3(n)}{n}$, where 
  $\gamma$ is a constant. 
  As we clarify in \cref{rem:time_is_flexible}, the choice of the exponent of the logarithm in the 
  definition of $t(n)$ is not rigid. 
  We choose the exponent $4$ instead of $3$ since it lets us get some estimates in a cleaner form and 
  makes it possible to avoid inserting a constant in the definition of $t(n)$.}
  hold
  \begin{itemize}
    \item We know the asymptotic behavior of the expected matching cost
    \begin{equation}\label{eq:matching_cost_limit}
      \lim_{n\to\infty} \E{W_2^2(\m, \mu^n)}\left(\frac1{4\pi}\frac{\log(n)}n\right)^{-1} = 1 \fullstop
    \end{equation}
    \item The probability of the complement of $A^{n,t}_\xi$ decays faster than any power.
    In formulas, for any $k>0$ there exists a constant $C=C(M, k)$ such that 
    \begin{equation}\label{eq:exceptional_set_rare}
      \P{\left(A^{n,t}_\xi\right)^\complement} 
      \le C(M, k) n^{-k} \fullstop
    \end{equation}
    \item One has the refined contractivity estimate\footnote{This does not follow from the well-known 
    contractivity property for the heat semigroup. Indeed the standard contractivity would yield 
    an estimate of order $t=\gamma\frac{\log^4(n)}{n}\gg \frac{\log(n)}{n}$
    and such magnitude is too 
    large for our purposes.}
    \begin{equation}\label{eq:diffusion_error}
      \E{W_2^2(\mu^n, \mu^{n,t})} \lesssim \frac{\log(\log(n))}n 
      \left(\ll \frac{\log(n)}{n}\right)\fullstop
    \end{equation}
    \item We are able to control the perturbation error with
    \begin{equation}\label{eq:perturbation_error}
      \E{W_2^2(\mu^{n,t}, \hat\mu^{n,t})} \lesssim \frac{1}{n\log(n)} 
      \left(\ll \frac{\log(n)}{n}\right)\fullstop
    \end{equation}
    \item \label{it:exp_is_optimal}
    When $n$ is sufficiently large, in the event $A^{n,t}_\xi$ the map $\exp(\nabla f^{n,t})$ is 
    optimal from $\m$ to $\hat\mu^{n,t}$.
  \end{itemize}
\end{proposition}
\begin{proof}
  All of these results are contained in \cite{ambrosio-glaudo2018} and thus we will only
  give a precise reference for them. All references are to propositions contained
  in \cite{ambrosio-glaudo2018}.
  
  The validity of \cref{eq:matching_cost_limit} is contained in Theorem 1.2.
  The fact that the event $A^{n,t}_\xi$ has overwhelming probability follows from Theorem 3.3.
  The refined contractivity estimate \cref{eq:diffusion_error} is Theorem 5.2.
  
  The estimate \cref{eq:perturbation_error} follows from Equation 6.2 and Lemma 3.14. 
  More specifically Equation 6.2 tells us that in the event $A^{n,t}_\xi$ it holds
  \begin{equation*}
    \E{W_2^2(\mu^{n,t}, \hat\mu^{n,t})} \lesssim \xi^2 \int_M\abs{\nabla f^{n,t}}^2\de\m
  \end{equation*}
  and Lemma 3.14 gives us the expected value of the Dirichlet energy of $f^{n,t}$. The behavior in the 
  complementary of $A^{n,t}_\xi$ can be ignored thanks to \cref{eq:exceptional_set_rare}.
  
  It remains to show that in the event $A^{n,t}_\xi$, the map $\exp(\nabla f^{n,t})$ is optimal. 
  This follows directly from \cite[Theorem 1.1]{Glaudo19}.
\end{proof}

\begin{remark}\label{rem:repeat_old_remark}
  Let us repeat the elementary observation made in \cite[Remark 5.3]{ambrosio-glaudo2018}, as it will 
  be useful.
  
  Let $X, Y$ be two random variables such that, in an event $E$, it holds $X\le Y$. Then
  \begin{equation*}
    \E{X} \le \E{Y} + (\norm{X}_\infty + \norm{Y}_\infty)\P{E^\complement} \fullstop
  \end{equation*}
  In particular, if the infinity norm of $X, Y$ is suitably controlled and the probability of $E^\complement$
  is exceedingly small, we can assume $\E{X}\le \E{Y}$ up to a small error.
  
  This observation allows us to restrict our study to the \emph{good} event $A^{n,t}_\xi$. Indeed all 
  quantities involved in our computations have at most polynomial growth, whereas 
  $\P{(A^{n,t}_\xi)^\complement}$ decays faster than any power.
\end{remark}

Once we have these results in our hands, the proof of the main theorem follows rather easily. Indeed
we just have to check that all assumptions of our stability result are satisfied.

\begin{proof}[Proof of \cref{thm:main_theorem}]
  Let us assume to be in the event $A^{n,t}_\xi$ with $\xi = \frac1{\log(n)}$.
  Hence, thanks to \cref{it:exp_is_optimal}, we can apply \cref{thm:optimal_map_stability} to the triple
  of measures $\nu=\m$, $\mu_1=\hat\mu^{n,t}$ and $\mu_2=\mu^n$ (with 
  $S=\exp(\nabla f^{n,t})$ and $T=T^n$). 
  We obtain
  \begin{align*}
    \int_M d^2(\exp(\nabla f^{n,t}), T^n)\de\m &\lesssim 
    W_2^2(\mu^n, \hat\mu^{n,t}) + W_2(\mu^n, \hat\mu^{n,t})W_2(\m, \hat\mu^{n,t}) \\
    &\lesssim 
    W_2^2(\mu^n, \hat\mu^{n,t}) + W_2(\mu^n, \hat\mu^{n,t})W_2(\m, \mu^n) 
    \fullstop
  \end{align*}
  
  Recalling \cref{rem:repeat_old_remark} and \cref{eq:exceptional_set_rare}, if we consider the expected 
  value we can apply the latter inequality as if it were true unconditionally and not only in the 
  event $A^{n,t}_\xi$. Thus, taking the expected value and applying Cauchy-Schwarz's inequality, we get
  \begin{equation*}
    \E{\int_M d^2(\exp(\nabla f^{n,t}), T^n)\de\m} \lesssim 
    \E{W_2^2(\mu^n, \hat\mu^{n,t})} 
    + \sqrt{\E{W^2_2(\mu^n, \hat\mu^{n,t})}\cdot \E{W^2_2(\m, \mu^n)}}
    \fullstop
  \end{equation*}
  The desired statement follows directly applying 
  \cref{eq:matching_cost_limit,eq:diffusion_error,eq:perturbation_error}.
\end{proof}

\begin{remark}\label{rem:time_is_flexible}
  It might seem that our choice of the time $t=\log^4(n)/n$ is a little arbitrary, and indeed it is. Any time
  $t=t(n)$ of order $\log^\alpha(n)/n$, for some $\alpha > 3$, would have worked flawlessly.
\end{remark}

It remains to justify \cref{rem:distance-tangent}. As already said, the desired estimate boils down to
the validity of 
\begin{equation}\label{eq:linf_distance_bounded}
  \P{\|d(T^{n},\id)\|_\infty > \varepsilon} \ll \frac{\log(n)}n
\end{equation}
for any fixed $\eps>0$.
The strategy of the proof is as follows. 
With \cref{lem:linf_l2_bound} (see also \cite[Lemma 4.1]{goldman2017}) we reduce the hard task of 
controlling the $L^\infty$-distance between $T^n$ and $\id$ to the easier task of controlling 
$W_2^2(\m,\mu^n)$. 
This latter estimate is then shown to be a consequence of \cref{eq:exceptional_set_rare}.

\begin{lemma}\label{lem:linf_l2_bound}
  Let $(M,\metric)$ be a $d$-dimensional compact Riemannian manifold (possibly with Lipschitz boundary) and 
  let $\m$ be the volume measure on $M$.
  
  If $T:M\to M$ is the optimal map with respect to the quadratic cost from $\m$ to $T_\#\m$,
  then one has
  \begin{equation*}
    \norm{d(\id, T)}_{L^\infty(M)}
    \lesssim \left(\int_M d^2(\id, T)\de\m\right)^{\frac1{d+2}} \fullstop
  \end{equation*}
\end{lemma}
\begin{proof}
  Since the map $T$ is optimal, its graph is essentially contained in $c$-cyclically monotone set 
  (see~\cite[Theorem 1.38]{Santambrogio15}). 
  More precisely, there exists a Borel set $C\subseteq M$ such that $\{(x,T(x)):\ x\in C\}$ is $c$-cyclically 
  monotone and $M\setminus C$ is $\m$-negligible. 
  We will reduce our considerations to points in $C$ in order to exploit the $c$-cyclical monotonicity.

  Let us fix a point $x_0\in C$ and let us define $\alpha \defeq \frac12 d(x_0, T(x_0))$.
  Let us define the point $p\in M$ as the middle point between $x_0$ and $T(x_0)$, that is 
  $d(x_0, p) = d(p, T(x_0)) = \alpha$. 
  Let us consider a point $x\in B(p, \eps\alpha)\cap C$ where $\eps>0$ is a small constant that will 
  be chosen a posteriori. Finally let us define $\beta\defeq d(x, T(x))$. We want to show that 
  $\beta$ cannot be much smaller than $\alpha$.
  \begin{figure}[htb]
    \centering
    % \documentclass[crop,tikz]{standalone}
% \begin{document}
\begin{tikzpicture}[scale=9]
\usetikzlibrary{backgrounds}
\coordinate (X0) at (0,0);
\coordinate (TX0) at (1,0.07);
\coordinate (X) at (0.55, 0.1);
\coordinate (TX) at (0.75, -0.05);

\draw[dashed] (X0) to [bend left] 
  coordinate[midway](P) coordinate[pos=0.2](FIRSTA) coordinate[pos=0.8](SECONDA) 
  (TX0);
\begin{scope}[on background layer]
  \draw[fill=gray!10] (P) circle (0.15);
\end{scope}
\draw[dashed] (X) to coordinate[pos=0.45](B) (TX);

\node at (X0) {\textbullet};
\node[above] at (X0) {$x_0$};
\node at (TX0) {\textbullet};
\node[above] at (TX0) {$T(x_0)$};
\node at (P) {\textbullet};
\node[above] at (P) {$p$};
\node at (X) {\textbullet};
\node[above] at (X) {$x$};
\node at (TX) {\textbullet};
\node[right] at (TX) {$T(x)$};
\node[above] at (FIRSTA) {$\alpha$};
\node[above] at (SECONDA) {$\alpha$};
\node[below] at (B) {$\beta$};
\node at (0.45, 0.27) {\footnotesize $B(p, \varepsilon\alpha)$};

\end{tikzpicture}
% \end{document}
    \caption{The points considered in the proof of of \cref{lem:linf_l2_bound}.}
  \end{figure}
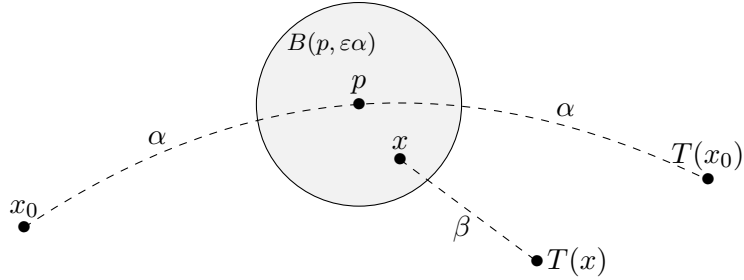

  Thanks to the $c$-cyclical monotonicity of $C$, it holds
  \begin{equation*}
    d^2(x_0, T(x_0)) + d^2(x, T(x)) \le d^2(x, T(x_0)) + d^2(x_0, T(x))
  \end{equation*}
  and thus, applying repeatedly the triangle inequality, we deduce
  \begin{align*}
    4\alpha^2 + \beta^2 
    &\le \left(d(x, p) + d(p, T(x_0))\right)^2 + \left(d(x_0, p) + d(p, x) + d(x, T(x))\right)^2 \\
    &\le (\eps\alpha + \alpha)^2 + (\alpha + \eps\alpha + \beta)^2 
    = 2(1+\eps)^2\alpha^2 + \beta^2 + 2(1+\eps)\alpha\beta \\
    &\hphantom{asdasdasdasdasdasdasd}\Updownarrow \\
    &\hphantom{asdasdasd}(2-(1+\eps)^2) \alpha \le (1+\eps)\beta \fullstop
  \end{align*}
  If $\eps$ is chosen sufficiently small (i.e. $\eps = 1/3$), the desired estimate
  $\alpha\lesssim\beta$ follows.
  
  Since $x$ can be chosen arbitrarily in $B(p, \eps\alpha)\cap C$, the estimate 
  $\alpha\lesssim \beta$ implies
  \begin{align*}
    \int_M d^2(x, T(x))\de\m(x) 
    &\ge \int_{B(p, \eps\alpha)} d^2(x, T(x))\de\m(x)
    \gtrsim \m(B(p, \eps\alpha)) d^2(x_0, T(x_0)) \\
    &\gtrsim \eps\alpha^d d^2(x_0, T(x_0)
    \gtrsim \bigl(d(x_0, T(x_0))\bigr)^{d+2} \comma
  \end{align*}
  where we have used that a ball with radius $r$ not larger than the diameter of 
  $M$ has measure comparable to $r^d$ (follows from the Ahlfors-regularity of compact 
  Riemannian manifolds with Lipschitz boundary).
  This completes the proof since $x_0$ can be chosen arbitrarily in a set with full measure.
\end{proof}
\begin{remark}
  The previous lemma holds, with the same proof, on any Ahlfors-regular metric measure space 
  that is also a length space.
\end{remark}
\begin{remark}
  If we apply \cref{lem:linf_l2_bound} on a $2$-dimensional manifold with $T^n$ being the optimal map
  (with respect to the quadratic cost) from $\m$ to the empirical measure $\mu^n$, we obtain
  \begin{equation*}
    \norm{d(\id,T^n}_{L^{\infty}(M)} 
    \lesssim W_2(\m, \mu^n)^{\frac{1}{2}} \fullstop
  \end{equation*}
  Since we know (as a consequence of \cref{eq:limit_value}) that with high probability
  $W^2_2(\m, \mu^n) \lesssim n^{-1}\log(n)$, we deduce that with high probability
  it holds
  \begin{equation*}
    \norm{d(\id,T^n}_{L^{\infty}(M)}  \lesssim \left(\frac{\log(n)}{n}\right)^{\frac14} \fullstop
  \end{equation*}
  This estimate does not match the asymptotic behavior of the $\infty$-Wasserstein distance between $\m$
  and $\mu^n$. In fact, as proven in \cite{leighton1989,shor1991,trillos2015}, with high probability 
  it holds
  \begin{equation*}
    W_{\infty}(\m, \mu^n) \approx \frac{\log(n)^{\frac34}}{n^{\frac12}} \fullstop
  \end{equation*}
\end{remark}

We are now ready to show \cref{eq:linf_distance_bounded} (to be precise we prove a much stronger 
estimate).
\begin{proposition}\label{prop:linf_is_small}
  Using the same notation and definitions of the statement of \cref{thm:main_theorem}, for any 
  $\eps>0$ and any $k>0$ there exists a constant $C=C(M, \eps, k)$ such that
  \begin{equation}
    \P{\|d(T^{n},\id)\|_\infty > \varepsilon} \le C(M, \eps, k)n^{-k} \fullstop
  \end{equation}
\end{proposition}
\begin{proof}
  We show that for any $\eps>0$ and any $k>0$ there exists a constant $C=C(M, \eps, k)$
  such that
  \begin{equation}\label{eq:wasserstein_is_small}
    \P{W_2(\m, \mu^n) > \varepsilon} \le C(M, \eps, k)n^{-k} \fullstop
  \end{equation}
  In fact, if we are able to prove \cref{eq:wasserstein_is_small}, then the statement of the proposition 
  follows applying \cref{lem:linf_l2_bound} with $T=T^n$ (changing adequately $\eps,k$ and the value
  of the constant $C$).
  
  The triangle inequality gives us
  \begin{equation}\label{eq:linf_tmp1}
    W_2(\m, \mu^n) \le W_2(\mu^{n,t}, \mu^n) + W_2(\m, \mu^{n,t}) \fullstop
  \end{equation}
  The first term can be bounded using the contractivity property of the heat semigroup, obtaining
  \begin{equation}\label{eq:linf_tmp2}
    W_2(\mu^{n,t}, \mu^n) \lesssim \sqrt{t} \fullstop
  \end{equation}
  For the second term we employ the transport inequality \cite[(4.1)]{ambrosio-glaudo2018} and get
  \begin{equation}\label{eq:linf_tmp3}
    W^2_2(\m, \mu^{n,t}) \lesssim \int_M \abs{\nabla f^{n,t}}^2\de\m \fullstop
  \end{equation}
  If we assume to be in the event $A^{n,t}_\xi$ (that is defined in the statement of \cref{prop:ag18})
  with $\xi=\xi(n)=\frac1{\log(n)}$, we have
  \begin{equation}\label{eq:linf_tmp4}
    \int_M \abs{\nabla f^{n,t}}^2\de\m \lesssim \norm{\nabla f}_{L^{\infty}(M)}^2
    \lesssim \norm{\nabla^2 f}_{L^{\infty}(M)}^2 \le \xi^2 \fullstop
  \end{equation}
  Joining \cref{eq:linf_tmp1,eq:linf_tmp2,eq:linf_tmp3,eq:linf_tmp4} we deduce that in the event 
  $A^{n,t}_\xi$ it holds
  \begin{equation*}
    W_2(\m, \mu^n) \lesssim \sqrt{t} + \xi \fullstop
  \end{equation*}
  Since $t(n)\to 0$ and $\xi(n)\to 0$ as $n\to\infty$, this implies (for $n$ sufficiently large) that
  in the event $A^{n,t}_\xi$ it holds $W_2(\m, \mu^n) \le \eps$. Hence \cref{eq:wasserstein_is_small}
  is a consequence of \cref{eq:exceptional_set_rare} and this concludes of the proof.
\end{proof}

\printbibliography

\end{document}